\documentclass{amsart}
\usepackage{amssymb,amsthm}
\usepackage{hyperref}
\allowdisplaybreaks

\newcommand{\I}{{\mathfrak{I}}}
\newcommand{\R}{{\mathbb{R}}}
\newcommand{\T}{{\mathbb{T}}}
\newcommand{\Z}{{\mathbb{Z}}}
\newcommand{\op}{\textit{op}}

\newcommand{\qtq}[1]{\quad\text{#1}\quad}
\newcommand{\eps}{\varepsilon}
\newcommand{\vk}{\varkappa}
\newcommand{\ddt}{\frac{d\ }{dt}}
\DeclareMathOperator{\sgn}{sgn}
\DeclareMathOperator{\tr}{tr}

\newcommand{\HH}{\mathsf{H}}

\newtheorem{theorem}{Theorem}[section]
\newtheorem{prop}[theorem]{Proposition}
\newtheorem{lemma}[theorem]{Lemma}

\theoremstyle{definition}
\newtheorem{definition}[theorem]{Definition}
\theoremstyle{remark}

\begin{document}

\title[A priori bounds and equicontinuity for ILW]{A priori bounds and equicontinuity of orbits\\for the intermediate long wave equation}
\author{B. Harrop-Griffiths, R. Killip, and M. Vi\c{s}an}

\address
{Benjamin Harrop-Griffiths\\
Department of Mathematics and Statistics\\
Georgetown University, DC 20057, USA}
\email{benjamin.harropgriffiths@georgetown.edu}

\address
{Rowan Killip\\
CEREMADE, CNRS\\Universit\'e Paris Dauphine--PSL\\ Place du
Mar\'echal de Lattre de Tassigny\\ 75016 Paris\\ France \& Department
of Mathematics\\ University of California\\Los Angeles\\CA 90095\\USA}
\email{killip@ceremade.dauphine.fr}

\address
{Monica Vi\c{s}an\\
Department of Mathematics\\
University of California, Los Angeles, CA 90095, USA}
\email{visan@math.ucla.edu}

\begin{abstract}
We prove uniform-in-time a priori $H^s$ bounds for solutions to the intermediate long wave equation posed both on the line and on the circle, covering the range $-\frac12<s\leq0$.   Additionally, we prove that the set of orbits emanating from a bounded and equicontinuous set in $H^s$ is also bounded and equicontinuous in $H^s$.  Our proof is based on the identification of a suitable Lax pair formulation for the intermediate long wave equation.
\end{abstract}

\maketitle

\section{Introduction}

The intermediate long wave equation describes the propagation of internal waves at the interface of a stratified fluid of total depth $h>0$.  It takes the form
\begin{align}\tag{ILW}\label{ILW}
\tfrac{d}{dt}q = - Tq'' - \tfrac1h q' -2qq',
\end{align}
with $T$ denoting the  Fourier multiplier operator with symbol
\begin{equation}\label{T defn}
\widehat {Tf}(\xi) = i\coth(h\xi) \widehat f(\xi).
\end{equation}
The unknown function $q:\R\times\mathcal M\to \R$ represents the elevation of the interface relative to equilibrium.  In this paper, we treat the cases $\mathcal M =\R$ and $\mathcal M =\T=\R/\Z$ on equal footing.

The equation \eqref{ILW} can be viewed as an `intermediate' equation between the Korteweg--de Vries equation (which arises in the shallow water limit $h\to 0$),
\begin{align}\tag{KdV}\label{KdV}
\tfrac{d}{dt}q = -q''' - 2qq',
\end{align}
and the Benjamin--Ono equation (which arises in the deep water limit $h\to \infty$),
\begin{align}\tag{BO}\label{BO}
\tfrac{d}{dt}q = \HH q'' -2qq'.
\end{align}
Here $\HH$ denotes the Hilbert transform, $\widehat {\HH f}(\xi) = -i\sgn(\xi) \widehat f(\xi)$.

Smooth solutions to \eqref{ILW} conserve the momentum,
\begin{equation}\label{moment}
M(q) = \tfrac12 \int q^2(t,x)\, dx,
\end{equation}
and the Hamiltonian,
\begin{equation}
H(q) = \tfrac12 \int q(t,x) (Tq')(t,x) \, dx + \tfrac1{2h} \int q^2(t,x) \, dx + \tfrac13 \int q^3(t,x)\, dx. 
\end{equation}

The central goal of this paper is to prove both a priori norm bounds and equicontinuity of orbits for solutions to \eqref{ILW}.  Evidently, this is predicated on the existence of solutions.  For this reason, we will now give a brief account of the existing results on the well-posedness problem for this model.   For a more comprehensive discussion of \eqref{ILW}, we recommend the recent book \cite{KSbook}.  

Global well-posedness of \eqref{ILW} in $H^{s}(\mathcal M)$ for $s>\frac32$ was proven in both geometries already in \cite{ABFS}.  This guarantees that smooth initial data leads to global smooth solutions.   Here and throughout the paper, smooth means $H^\infty(\mathcal M)$. 

In \cite{MV}, Molinet and Vento proved global well-posedness of \eqref{ILW} in $H^{1/2}(\mathcal M)$ via the energy method. In \cite{MPV}, Molinet, Pilod, and Vento combined the energy method in \cite{MV} with Strichartz estimates and proved local well-posedness of \eqref{ILW} in $H^s(\R)$ for $s>\frac14$.

By treating \eqref{ILW} as a perturbation of the Benjamin--Ono equation and employing ideas from \cite{IT}, Ifrim and Saut \cite{IS} proved global well-posedness for \eqref{ILW} in $L^2(\mathcal \R)$.  Chapouto, Li, Oh, and Pilod \cite{CLOP} proved global well-posedness for \eqref{ILW} in $L^2(\mathcal M)$, treating both geometries in a uniform manner.  They likewise treated \eqref{ILW} as a perturbation of Benjamin--Ono, employing the  gauge transform of Tao and ideas from \cite{MP}.  Using a similar perturbative approach, but this time relying on Birkhoff coordinates, Gassot and Laurens \cite{GL} very recently proved global well-posedness in \(H^s(\T)\) for all \(s>-\frac12\).  This result is most likely sharp in the scale of Sobolev spaces: It was proved in \cite{CFLOP} that \eqref{ILW} is ill-posed in \(H^s(\mathcal M)\) for all \(s<-\frac12\) and in \cite{MR4652410} that \eqref{BO} is ill-posed in \(H^{-\frac12}(\T)\).

Chapouto, Forlano, Li, Oh, and Pilod \cite{CFLOP} obtained growth bounds for smooth solutions to \eqref{ILW} in Sobolev spaces $H^s(\mathcal M)$ for $-\frac12<s<0$. They achieved this by treating \eqref{ILW} as a perturbation of \eqref{BO} and employing the uniform-in-time a priori bounds for solutions to \eqref{BO} proved in \cite{KLV}. Specifically, they showed that for all $-\frac12<s<0$ and $\eps>0$, there exist constants $A_{h,s,\eps} \simeq h^{-2} (1+h^{-(\frac12+|s|+\eps)})$ and $C_s>0$, so that smooth solutions to \eqref{ILW} satisfy
\begin{align}\label{grow bdds}
\|q(t)\|_{H^s} \leq C_s^{1+|s|} e^{A_{h,s,\eps}|t|} \bigl( 1+2C_se^{A_{h,s,\eps}|t|} \|q(0)\|_{H^s}\bigr)^{\frac{2|s|}{1-2|s|}}\|q(0)\|_{H^s}
\end{align}
for all $t\in \R$.

Our first result in this paper is an improvement over \eqref{grow bdds} in the form of \emph{uniform-in-time} a priori bounds for smooth solutions to \eqref{ILW}.  Moreover, we prove for the first time that the set of orbits emanating from a bounded and equicontinuous set in $H^s(\mathcal M)$ is also bounded and equicontinuous in $H^s(\mathcal M)$:

\begin{theorem}\label{T1} Fix $0<h_0<\infty$ and $-\frac12< s\leq 0$.  For any smooth solution $q$ to \eqref{ILW}, we have
\begin{align}\label{a priori bdds}
\| q(0) \|_{H^s} \bigl( 1 +  \| q(0) \|_{H^s}\bigr)^{-2|s|} \lesssim \| q(t) \|_{H^s} \lesssim \| q(0) \|_{H^s} \bigl( 1 +  \| q(0) \|_{H^s}\bigr)^{\frac{2|s|}{1-2|s|}},
\end{align}
uniformly for all \(h\geq h_0\).

Moreover, if $Q\subset H^\infty(\mathcal M)$ is bounded and equicontinuous in $H^{s}(\mathcal M)$ and \(Q_h^*\) denotes the set of orbits emanating from $Q$ under the \eqref{ILW} flow, the set \(\bigcup_{h\geq h_0}Q_h^*\) forms an equicontinuous set in $H^{s}(\mathcal M)$.
\end{theorem}

We recall that the equation \eqref{ILW} converges to \eqref{BO} in the deep water limit \(h\to \infty\).  This convergence was proved to hold locally in time for solutions in \(H^s(\R)\) when \(s\geq 0\) in \cite{CLOP} and in \(H^s(\T)\) when \(s>-\frac12\) in \cite{GL}; see also \cite{ABFS,MR4866612,MR4752554}.  Theorem~\ref{T1} proves that both a priori bounds and preservation of equicontinuity hold uniformly in time in the deep water limit, providing strong evidence that \eqref{ILW} should converge locally in time to \eqref{BO} as \(h\to\infty\) also in \(H^s(\R)\) for all \(s>-\frac12\).

To address the shallow water regime $h\to 0$, if \(q\) is a solution of \eqref{ILW} we let
\[
u(t,x) := \tfrac3hq\bigl(\tfrac{3t}h,x\bigr).
\]
Then \(u\) is a solution of
\begin{equation}\label{ILW'}\tag{ILW$'$}
\tfrac{d}{dt}u = - \tfrac3hTu'' - \tfrac3{h^2} u' -2uu',
\end{equation}
which formally converges to \eqref{KdV} in the shallow water limit \(h\to0\).  The validity of this limit was proved locally in time in \(H^s(\mathcal M)\) for all \(s>\frac12\) in \cite{MR4752554}, improving the previous result of \cite{ABFS}.  We also mention that convergence of the Gibbs measures in both the deep and shallow water limits were considered in \cite{MR4878416}.

Our second result establishes an analog of Theorem~\ref{T1} that holds uniformly in the shallow water limit:

\begin{theorem}\label{T2}
Fix $-\frac12< s <0$ and \(K>0\).  Then, for any smooth solution $u$ to \eqref{ILW'} satisfying \(\|u(0)\|_{H^s}\leq K\) there exists \(h_0 = h_0(s,K)\) so that
\begin{align}\label{a priori bdds''}
\| u(0) \|_{H^s} \bigl( 1 +  \| u(0) \|_{H^s}\bigr)^{-\frac{2|s|}{\frac32+|s|}} \lesssim \| u(t) \|_{H^s} \lesssim \| u(0) \|_{H^s} \bigl( 1 +  \| u(0) \|_{H^s}\bigr)^{\frac{2|s|}{\frac32-|s|}},
\end{align}
uniformly for all \(0<h\leq h_0\).

Moreover, if $Q\subset H^\infty(\mathcal M)$ is bounded and equicontinuous in $H^{s}(\mathcal M)$ so that \(\|u\|_{H^s}\leq K\) for all \(u\in Q\) and \(Q_h^*\) denotes the set of orbits emanating from $Q$ under the \eqref{ILW'} flow, the set \(\bigcup_{0<h\leq h_0}Q_h^*\) is equicontinuous in $H^{s}(\mathcal M)$.
\end{theorem}

The complete integrability of \eqref{ILW} has been known for some time.  As well as constructing an infinite family of conservation laws, the paper \cite{MR0591811} introduced a zero-curvature representation.  Concretely, \eqref{ILW} was shown to be the consistency condition for a coupled system of equations describing the spatial and temporal derivatives of a pair of functions, namely, the upper and lower boundary values of a function holomorphic in a strip.
The subsequent paper \cite{YAS} then proposed an inverse scattering approach based on this representation.

This complicated formulation is at odds with the direct approach of \cite{KVZ} that we wish to employ.  Instead, we would like a Lax pair representation of a much more classical form: an isospectral deformation of an operator (ideally self-adjoint) acting on a concrete Hilbert space.  Taking inspiration from \cite{YAS,MR0591811}, we introduce just such a traditional Lax representation:  Defining
\begin{equation}\label{LP}\begin{aligned}
L(t) &:= -i\partial_x + \tfrac1{2h}(e^{2ih\partial_x}-1) - q(t,x), \\
P(t) &:= -\tfrac1h \partial_x -i \partial_x^2  + \tfrac1h\partial_x e^{2ih\partial_x} - (iTq')(t,x) - (\partial_x q(t,x) +q(t,x)\partial_x),
\end{aligned}\end{equation}
we have that
\begin{align}
\text{$q(t)$ solves \eqref{ILW}} \iff \tfrac{d}{dt} L(t) = [P(t),\, L(t)].        \label{Lax}
\end{align}
The validity of \eqref{Lax} will be verified in Section~\ref{S:LP}.  There we will also explain why $L$ is self-adjoint, even for $q\in H^s(\mathcal M)$ with $-\frac12<s\leq 0$.  Although it is merely incidental to our main development, we will also exhibit an illusory Lax pair for \eqref{ILW} and a number of related models in Lemma~\ref{L:IllLP}.  After completing this manuscript, we discovered that the Lax pair \eqref{LP} had previously appeared in \cite{MR597631}.

The free Lax operator (that corresponds to $q\equiv 0$) is simply the Fourier multiplier with symbol
\begin{equation}\label{symbol}
a_h(\xi) = \xi+ \tfrac1{2h}(e^{-2h\xi}-1).
\end{equation}
This function is non-negative.
In particular, the free resolvent,
$$
R_0(\kappa):=\bigl(-i\partial_x + \tfrac1{2h}(e^{2ih\partial_x}-1)+ \kappa\bigr)^{-1},
$$
exists as a positive definite $L^2$-bounded operator for all positive $\kappa$.

Mirroring the treatment of KdV in \cite{KVZ}, the central object of our analysis is the logarithm of the renormalized (in the sense of \cite{Hilbert}) perturbation determinant; however, as there, it is more convenient to introduce this quantity directly as a series of traces:
\begin{align}\label{KdValpha}
\alpha(\kappa;q) := \sum_{\ell=2}^\infty \frac{1}{\ell} \tr\Bigl\{ \bigl( \sqrt{R_0(\kappa)}\, q\, \sqrt{R_0(\kappa)}\bigr)^\ell \Bigr\}.
\end{align}

The convergence of this series for $\kappa$ large follows from Lemmas~\ref{L:dumb} and~\ref{L:HS}.  We prove that \eqref{KdValpha} is conserved under the \eqref{ILW} flow in Proposition~\ref{P:cons}.  It then remains to show that this quantity encodes conservation laws and equicontinuity at regularities $-\frac12<s\leq 0$.  The central idea is that for $\kappa$ very large, the series \eqref{KdValpha} can be conflated with its first term.  This in turn connects to the distribution of $q$ across frequencies; see \eqref{I2}.  The rather inexplicit form \eqref{F defn} of the function $F(\xi;\kappa)$ appearing there introduces some technical challenges, particularly for the question of $L^2$-equicontinuity; see Lemma~\ref{L:F}.

The technology for demonstrating equicontinuity of orbits was present already in \cite{KVZ} (also see \cite{KT}); however, this was neither appreciated nor exploited until \cite{KV}.  The utility of such equicontinuity results in treating the well-posedness problem has since been repeatedly demonstrated; see, for example, \cite{BKV:KdV5,Forlano,HKVZ,HGKV:NLS,HGKNV:DNLS,KLV,KLV:CCM,KNV,KKL,KV}. 

\subsection*{Acknowledgements} B.~H.-G. was supported by NSF grant DMS-2406816, R. K. was supported by NSF grant DMS-2154022, and M.~V. by NSF grant DMS-2348018.  We are grateful to Yilun Wu and Katie Marsden for discussions about this work.  We also wish to thank the anonymous referees for their helpful comments.

\subsection*{Data availability statement} There is no data associated with this article.

\section{Notation and Preliminaries}\label{S:2}

We write $A\lesssim B$ to indicate $A\leq CB$ for an absolute constant $C>0$ whose specific value is not important.  Occasionally, we use subscripts to indicate dependence of the constant $C$ on other parameters; for instance, we write $A \lesssim_{\alpha, \beta} B$ when $A \leq CB$ with the constant $C>0$ depending on $\alpha, \beta$.  We write $A\simeq B$ when $A\lesssim B$ and $B\lesssim A$.

Our conventions for the Fourier transform are as follows:
\begin{align*}
\widehat f(\xi) = \tfrac{1}{\sqrt{2\pi}} \int_\R e^{-i\xi x} f(x)\,dx  \qtq{so} f(x) = \tfrac{1}{\sqrt{2\pi}} \int_\R e^{i\xi x} \widehat f(\xi)\,d\xi
\end{align*}
for functions on the line \(\R\) and
\begin{align*}
\widehat f(\xi) = \int_0^1 e^{- i\xi x} f(x)\,dx \qtq{so} f(x) = \sum_{\xi\in 2\pi\Z} \widehat f(\xi) e^{i\xi x}
\end{align*}
for functions on the circle $\T$.  These definitions of the Fourier transform are unitary on $L^2$ and yield the Plancherel identities
\begin{align*}
    \|f\|_{L^2(\mathbb R)}=\|\widehat f\|_{L^2(\mathbb R)} \qtq{and}   \|f\|_{L^2(\mathbb T)}=\sum_{\xi\in 2\pi \mathbb Z}|\widehat f(\xi)|^2,
\end{align*}
as well as the following convolution identity on $\R$:
\begin{align*}
    \widehat {fg}= \tfrac{1}{\sqrt{2\pi}} \widehat f \ast \widehat g.
\end{align*}

For $\kappa\geq 1$, we define the Sobolev spaces $H^s_\kappa$ via the norms
\begin{align*}
\| f\|_{H^{s}_\kappa(\R)}^2 = \int_\R (\kappa^2+\xi^2)^s |\widehat f(\xi)|^2  \,d\xi
\end{align*}
and
\begin{align*}
\| f\|_{H^{s}_\kappa(\T)}^2 = \sum_{\xi\in 2\pi\Z} (\kappa^2+\xi^2)^s |\widehat f(\xi)|^2.
\end{align*}

\begin{definition}[Equicontinuity]\label{d:equicontinuity}
Fix $\sigma\in \R$. A bounded set \(Q\subset H^\sigma(\mathcal M)\) is said to be \emph{equicontinuous} if
\[
\limsup_{\delta\to0} \ \sup_{q\in Q} \ \sup_{|y|<\delta}\|q(\cdot +y) - q(\cdot)\|_{H^\sigma} = 0.
\]
\end{definition}

By Plancherel, equicontinuity in the spatial variable is equivalent to tightness in the Fourier variable. Specifically, a bounded set $Q\subset H^\sigma(\mathcal M)$ is equicontinuous if and only if
\begin{alignat*}{2}
\lim_{ \kappa\to\infty}\sup_{q\in Q}\int_{|\xi|\geq \kappa} & |\widehat{q}(\xi)|^2 (|\xi|+1)^{2\sigma}\, d\xi =0,&\quad &\text{when $\mathcal M = \R$,} \quad\text{or}\\
\lim_{ \kappa\to\infty}\, \sup_{q\in Q}\ \sum_{|\xi|\geq \kappa} \ & |\widehat{q}(\xi)|^2 (|\xi|+1)^{2\sigma}=0,&\quad&\text{when $\mathcal M = \T$}.
\end{alignat*}
This characterization leads quickly to a description of equicontinuity in terms of the Sobolev spaces $H^s_\kappa$:

\begin{lemma}[Characterization of equicontinuity, {\cite[Lemma~2.4]{KLV}}]\label{L:equi}
Fix $-\frac12<s<0$ and let $Q$ be a bounded subset of $H^s(\mathcal M)$.  The following are equivalent:\\[1mm]
(i) The subset $Q$ is equicontinuous in $H^s$.\\[1mm]
(ii) $\|q\|_{H^s_\kappa} \to 0$ as $\kappa\to\infty$ uniformly for $q\in Q$.
\end{lemma}

The book \cite{MR2154153} provides an excellent introduction to the trace functional on Hilbert space and the associated ideals $\I_p$, which play the role of noncommutative $L^p$ spaces.  Here we just give a brief account of the facts we need.

For an operator $A$ belonging to trace class and having a continuous integral kernel $K(x,y)$, the trace may be computed via
\begin{equation}\label{trace}
\tr(A) = \int  K(x,x) \,dx.
\end{equation}

Trace class operators are precisely those that can be written as the product of two Hilbert--Schmidt operators.  An operator $A$ is Hilbert--Schmidt if and only if its integral kernel $K(x,y)$ belongs to $L^2(\mathcal M\times\mathcal M)$; indeed, its Hilbert--Schmidt norm is then given by
\begin{equation}\label{HS norm}
\| A \|_{\mathfrak I_2}^2 =  \tr(A^* A) = \iint | K(x,y)|^2 \,dx\,dy.
\end{equation}

For much of what follows, we will need only the following basic assertions.  As illustrated in \cite{KVZ}, these can be verified directly, without recourse to the general theory.

\begin{lemma}[{\cite[Lemma 1.4]{KVZ}}] \label{L:dumb}
Let $A_i$ denote Hilbert--Schmidt operators on $L^2(\mathcal M)$.  Then
\begin{align}\label{dumb norm}
\| A_i \|_{\op} &\leq \| A_i \|_{\mathfrak I_2} \\
| \tr( A_1\cdots A_\ell ) | &\leq \prod_{i=1}^\ell \| A_i \|_{\mathfrak I_2}\quad\text{for all integers $\ell\geq 2$.} \label{dumb holder}
\end{align}
\end{lemma}

In order to verify the convergence of \eqref{KdValpha} using \eqref{dumb holder}, we need a suitable Hilbert--Schmidt estimate.  This is the topic of our next lemma.

\begin{lemma}\label{L:HS}
For $h,\kappa>0$, we have
\begin{align}\label{I2}
 \bigl\| \sqrt{R_0(\kappa)}\, q\, \sqrt{R_0(\kappa)} \bigr\|_{\mathfrak I_2(\mathcal M)}^2
 = \begin{cases}\displaystyle\int_{\R}  F(\xi;\kappa,h)\, |\widehat q(\xi)|^2\,d\xi &\quad\text{if \(\mathcal M = \R\),}\vspace{1em}\\ \displaystyle\sum_{\xi\in2\pi\Z} F(\xi;\kappa,h)\, |\widehat q(\xi)|^2&\quad\text{if \(\mathcal M = \T\),}\end{cases}
\end{align}
where 
\begin{align}\label{F defn}
 F(\xi;\kappa,h) &:= \begin{cases}
 	\displaystyle\ \int_{\R} \bigl[a_h(\eta) +\kappa\bigr]^{-1}\bigl[a_h(\xi+\eta) + \kappa\bigr]^{-1} \,\tfrac{d\eta}{2\pi}  & \quad\text{on $\R$,}\\[2ex]
	\displaystyle\sum_{\eta\in2\pi\Z} \bigl[a_h(\eta) +\kappa\bigr]^{-1}\bigl[a_h(\xi+\eta) + \kappa\bigr]^{-1} & \quad\text{on $\T$,}
\end{cases}
\end{align}
with \(a_h(\xi)\) defined as in \eqref{symbol}.

The function \(F(\xi;\kappa,h)\) is smooth, non-negative, even in \(\xi\), and satisfies the estimate
\begin{equation}\label{F main bound}
F(\xi;\kappa,h)\simeq \bigl[\tfrac{h\xi^2}{1+h|\xi|} +\kappa\bigr]^{-1}\Bigl[\sqrt{\tfrac{1+h\kappa}{h\kappa}} + \log\bigl(1 + \tfrac{h|\xi|}{1+h\kappa}\bigr)\Bigr].
\end{equation} 

In particular, if $-\frac12<s\leq 0$ then
\begin{align}\label{HS bdd}
 \| \sqrt{R_0(\kappa)}\, q\, \sqrt{R_0(\kappa)}\|_{\mathfrak I_2(\mathcal M)}^2 \lesssim_s\begin{cases} \kappa^{-(1-2|s|)} \|q\|_{H^s_\kappa(\mathcal M)}^2&\quad\text{if \(\kappa>\tfrac1h\),}\vspace{1em}\\\bigl(\tfrac \kappa h\bigr)^{-(\frac32-|s|)}\bigl\|\tfrac qh\bigr\|_{H_{\sqrt{\kappa/h}}^s(\mathcal M)}^2&\quad\text{if \(0<\kappa\leq\tfrac1h\).}\end{cases}
\end{align}

Consequently, there exists a constant \(C_s>0\) so that if \(h,\delta>0\) and
\begin{equation}\label{k bigger}
\kappa \geq \max\Bigl\{\tfrac1h,\bigl[ 1 + \tfrac1{\delta^2} C_s \| q\|_{H^s}^2 \bigr]^{\frac{1}{1-2|s|}}\Bigr\}
\end{equation}
or
\begin{equation}\label{shallow k bigger}
\bigl[ 1 + \tfrac1{\delta^2} C_s\bigl\|\tfrac qh\bigr\|_{H^s}^2 \bigr]^{\frac{1}{\frac32-|s|}}\leq \tfrac\kappa h\leq \tfrac1{h^2}
\end{equation}
then
\begin{align}\label{small norm}
 \| \sqrt{R_0(\kappa)}\, q\, \sqrt{R_0(\kappa)}\|_{\mathfrak I_2(\mathcal M)}^2 <\delta^2.
\end{align}
\end{lemma}
\begin{proof}
The identities \eqref{I2} and \eqref{F defn} follow from considering the Fourier kernel of the operator \(\sqrt{R_0(\kappa)}\, q\, \sqrt{R_0(\kappa)}\).

Making the change of variables \(\eta\mapsto \eta-\xi\) in \eqref{F defn} proves that the map \(\xi\mapsto F(\xi;\kappa,h)\) is even.  That $F$ is also smooth and non-negative follows from the fact that the function $a_h$ is smooth and non-negative.

The estimate \eqref{HS bdd} follows readily from \eqref{F main bound}, while \eqref{small norm} follows from \eqref{HS bdd} under the assumptions \eqref{k bigger} and \eqref{shallow k bigger}.  Thus, it remains to estimate the size of $F(\xi;\kappa,h)$.  We present the details on the real line; the periodic setting can be treated analogously, with integrals being replaced by sums.  Further, as \(F\) is even in \(\xi\) and $|\widehat q(\xi)| = |\widehat q(-\xi)|$ (since $q$ is real-valued), it suffices to consider the case that \(\xi\geq0\). 

Denoting \(a(\xi) = a_1(\xi) = \xi + \tfrac12(e^{-2\xi}-1)\), we compute that \(a'(\xi) = 1 - e^{-2\xi}\) and \(a''(\xi) = 2e^{-2\xi}>0\), so \(a(\xi)\) is convex with a global minimum at \(\xi = 0\). By considering the asymptotic behavior of \(a(\xi)\) as \(\xi\to0\) and \(\xi\to \pm\infty\), we see that
\begin{equation}\label{a bounds}
a(\xi)\gtrsim \tfrac{\xi^2}{1+|\xi|}\quad\text{for all \(\xi\in \R\)}\quad\text{and}\quad a(\xi)\simeq \tfrac{\xi^2}{1+\xi}\quad\text{for all \(\xi\geq0\).}
\end{equation}

As \(a_h(\xi) = \frac1ha(h\xi)\), the estimate \eqref{a bounds} yields the (uniform in \(h\)) upper bound
\begin{align}
F(\xi;\kappa,h) &\lesssim \int_\R \bigl[\tfrac{h\eta^2}{1+h|\eta|} +\kappa\bigr]^{-1}\bigl[\tfrac{h(\xi+\eta)^2}{1+h|\xi+\eta|} + \kappa\bigr]^{-1} \,d\eta\notag\\
&\lesssim \int_{-\frac12\xi}^{\frac12\xi} \bigl[\tfrac{h\eta^2}{1+h|\eta|} +\kappa\bigr]^{-1}\bigl[\tfrac{h(\xi+\eta)^2}{1+h|\xi+\eta|} + \kappa\bigr]^{-1} \,d\eta\label{F upper bound}\\
&\quad  + \int_{\frac12\xi}^\infty \bigl[\tfrac{h\eta^2}{1+h|\eta|} +\kappa\bigr]^{-1}\bigl[\tfrac{h(\xi+\eta)^2}{1+h|\xi+\eta|} + \kappa\bigr]^{-1} \,d\eta,\notag
\end{align}
where the second inequality follows from the fact that the map
\[
\eta\mapsto \bigl[\tfrac{h\eta^2}{1+h|\eta|} +\kappa\bigr]^{-1}\bigl[\tfrac{h(\xi+\eta)^2}{1+h|\xi+\eta|} + \kappa\bigr]^{-1}
\]
is symmetric about \(\eta = -\frac12\xi\).

For the first summand on RHS\eqref{F upper bound}, we estimate
\begin{align*}
\int_{-\frac12\xi}^{\frac12\xi}\bigl[\tfrac{h\eta^2}{1+h|\eta|} +\kappa\bigr]^{-1}\bigl[\tfrac{h(\xi+\eta)^2}{1+h|\xi+\eta|} + \kappa\bigr]^{-1} \,d\eta \simeq \bigl[\tfrac{h\xi^2}{1+h|\xi|} +\kappa\bigr]^{-1}\int_0^{\frac12\xi}\bigl[\tfrac{h\eta^2}{1+h|\eta|} +\kappa\bigr]^{-1}\,d\eta.
\end{align*}
We then consider several cases depending on the relative size of \(\xi,\kappa,h\).

First, suppose that \(\kappa>\frac1h\). If additionally \(0\leq\xi\leq \kappa\), then
\begin{align*}
\int_0^{\frac12\xi}\bigl[\tfrac{h\eta^2}{1+h|\eta|} +\kappa\bigr]^{-1}\,d\eta&\simeq \int_0^{\frac12\xi}\tfrac1\kappa\,d\eta\simeq \tfrac\xi\kappa\lesssim 1,
\end{align*}
whereas if \(\xi>\kappa\) we have
\begin{align*}
\int_0^{\frac12\xi}\bigl[\tfrac{h\eta^2}{1+h|\eta|} +\kappa\bigr]^{-1}\,d\eta\simeq \int_0^{\frac12\kappa}\tfrac1\kappa\,d\eta + \int_{\frac12\kappa}^{\frac12\xi}\tfrac1\eta\,d\eta\simeq 1 + \log\bigl(\tfrac\xi\kappa\bigr).
\end{align*}

Second, suppose that \(0<\kappa\leq \frac1h\). If \(0\leq\xi\leq \sqrt{\frac\kappa h}\) then
\[
\int_0^{\frac12\xi}\bigl[\tfrac{h\eta^2}{1+h|\eta|} +\kappa\bigr]^{-1}\,d\eta \simeq \int_0^{\frac12\xi}\tfrac1\kappa\,d\eta \simeq \tfrac\xi\kappa\lesssim \tfrac1{\sqrt{h\kappa}},
\]
if \(\sqrt{\frac\kappa h}<\xi\leq\frac1h\) then
\begin{align*}
\int_0^{\frac12\xi}\bigl[\tfrac{h\eta^2}{1+h|\eta|} +\kappa\bigr]^{-1}\,d\eta&\simeq \int_0^{\frac12\sqrt{\frac\kappa h}}\tfrac1\kappa\,d\eta + \int_{\frac12\sqrt{\frac\kappa h}}^{\frac12\xi}\tfrac1{h\eta^2}\,d\eta\simeq \tfrac1{\sqrt{h\kappa}},
\end{align*}
and if \(\xi>\frac1h\) then
\begin{align*}
\int_0^{\frac12\xi}\bigl[\tfrac{h\eta^2}{1+h|\eta|} +\kappa\bigr]^{-1}\,d\eta&\simeq \int_0^{\frac12\sqrt{\frac\kappa h}}\tfrac1\kappa\,d\eta + \int_{\frac12\sqrt{\frac\kappa h}}^{\frac1{2h}}\tfrac1{h\eta^2}\,d\eta + \int_{\frac1{2h}}^{\frac12\xi}\tfrac1\eta\,d\eta\simeq \tfrac1{\sqrt{h\kappa}} + \log(h\xi).
\end{align*}

Putting these cases together, we arrive at the estimate
\[
\int_0^{\frac12\xi} \bigl[\tfrac{h\eta^2}{1+h|\eta|} +\kappa\bigr]^{-1} \,d\eta\lesssim \sqrt{\tfrac{1+h\kappa}{h\kappa}} + \log\bigl(1 + \tfrac{h\xi}{1+h\kappa}\bigr),
\]
which yields an acceptable upper bound for the first summand on RHS\eqref{F upper bound}. We also note that these estimates prove that
\begin{equation}\label{lower bound helper}
\int_0^{\frac12\max\{\xi,\kappa,\sqrt{\frac\kappa h}\}} \bigl[\tfrac{h\eta^2}{1+h|\eta|} +\kappa\bigr]^{-1} \,d\eta\simeq \sqrt{\tfrac{1+h\kappa }{h\kappa}} + \log\bigl(1 + \tfrac{h\xi}{1+h\kappa}\bigr),
\end{equation}
which we will use in our proof of the lower bound below.

Turning to the second summand on RHS\eqref{F upper bound}, we estimate
\[
\int_{\frac12\xi}^\infty \bigl[\tfrac{h\eta^2}{1+h|\eta|} +\kappa\bigr]^{-1}\bigl[\tfrac{h(\xi+\eta)^2}{1+h|\xi+\eta|} + \kappa\bigr]^{-1} \,d\eta\simeq \int_{\frac12\xi}^\infty \bigl[\tfrac{h\eta^2}{1+h|\eta|} +\kappa\bigr]^{-2}\,d\eta.
\]
We proceed as for the first summand: When \(\kappa>\frac1h\), we consider the cases \(\xi>\kappa\) and \(0\leq\xi\leq \kappa\) to obtain the estimate
\[
\int_{\frac12\xi}^\infty \bigl[\tfrac{h\eta^2}{1+h|\eta|} +\kappa\bigr]^{-2}\,d\eta\lesssim \bigl[\tfrac{h\xi^2}{1+h\xi} +\kappa\bigr]^{-1}\Bigl[\sqrt{\tfrac{1+h\kappa}{h\kappa}} + \log\bigl(1 + \tfrac{h\xi}{1+h\kappa}\bigr)\Bigr].
\]
Similarly, when \(0<\kappa\leq \frac1h\) we consider the cases \(\xi>\frac1h\), \(\sqrt{\frac\kappa h}<\xi\leq \frac1h\), and \(0\leq\xi\leq \sqrt{\frac\kappa h}\) to bound
\[
\int_{\frac12\xi}^\infty \bigl[\tfrac{h\eta^2}{1+h|\eta|} +\kappa\bigr]^{-2}\,d\eta\lesssim \bigl[\tfrac{h\xi^2}{1+h\xi} +\kappa\bigr]^{-1}\Bigl[\sqrt{\tfrac{1+h\kappa }{h\kappa}} + \log\bigl(1 + \tfrac{h\xi}{1+h\kappa}\bigr)\Bigr].
\]

Combining our estimates for the two summands on RHS\eqref{F upper bound} and recalling that \(F\) is even in \(\xi\) gives us the upper bound in \eqref{F main bound}.

Turning to the lower bound in \eqref{F main bound}, we take \(\xi\geq 0\) and combine \eqref{F defn} and \eqref{a bounds} to obtain
\begin{align*}
F(\xi;\kappa,h) &\gtrsim \int_0^{\frac12\max\{\xi,\kappa,\sqrt{\frac\kappa h}\}} \bigl[\tfrac{h\eta^2}{1+h|\eta|} +\kappa\bigr]^{-1}\bigl[\tfrac{h(\xi+\eta)^2}{1+h(\xi+\eta)} + \kappa\bigr]^{-1} \,d\eta\\
&\gtrsim \bigl[\tfrac{h\xi^2}{1+h\xi} + \kappa\bigr]^{-1}\int_0^{\frac12\max\{\xi,\kappa,\sqrt{\frac\kappa h}\}} \bigl[\tfrac{h\eta^2}{1+h|\eta|} +\kappa\bigr]^{-1} \,d\eta.
\end{align*}
Applying \eqref{lower bound helper} and again using that \(F\) is even in \(\xi\) completes the proof of \eqref{F main bound}. 
\end{proof}

Lemma~\ref{L:HS} guarantees convergence of the series defining $\alpha(\kappa;q(t))$ whenever $\kappa>0$ is sufficiently large depending on \(\|q(t)\|
_{H^s},h,s\).  To ensure that \(\kappa\) can be chosen uniformly in time, we will employ a continuity argument, the basis of which is formalized in the next lemma: 

\begin{lemma}[{\cite[Lemma 1.5]{KVZ}}] \label{L:C1}
Let $t\mapsto A(t)$ define a $C^1$ curve in $\mathfrak{I}_2$.  Suppose
\begin{align*}
\bigl\| A(t_0) \bigr\|_{\mathfrak I_2} < \tfrac13.
\end{align*}
Then there is a closed neighborhood $I$ of $t_0$ on which the series
\begin{align*}
\alpha(t) := \sum_{\ell=2}^\infty \tfrac{1}{\ell} \tr\bigl\{ A(t)^\ell \bigr\}
\end{align*}
converges and defines a $C^1$ function with
 \begin{align*}
\tfrac{d\ }{dt} \alpha(t) := \sum_{\ell=2}^\infty  \tr\bigl\{ A(t)^{\ell-1} \tfrac{d\ }{dt} A(t) \bigr\}.
\end{align*}
Moreover, if $A(t)$ is self-adjoint, then
\begin{align*}
\tfrac13 \| A(t) \|^2_{\mathfrak I_2} \leq \alpha(t) \leq \tfrac23 \| A(t) \|^2_{\mathfrak I_2} \quad\text{for all $t\in I$.}
\end{align*}
\end{lemma}

\section{Verification of the Lax pair}\label{S:LP}

Suppose $q\equiv 0$. Fourier analysis easily allows us to interpret the free Lax operator $L_0= -i\partial_x +\frac1{2h}(e^{2ih\partial_x}-1)$ as a self-adjoint operator that is bounded below.  In particular, the quadratic form domain of this operator is comprised of those $f\in L^2(\mathcal M)$ with 
\begin{equation}\label{Q(L0)}\begin{aligned}
\int_\R \bigl[\xi+ \tfrac1{2h}(e^{-2h\xi}-1) + 1\bigr] |\widehat f(\xi)|^2\,d\xi &< \infty \qtq{or}\\
\sum_{\xi\in2\pi\Z}  \bigl[\xi+ \tfrac1{2h}(e^{-2h\xi}-1) + 1\bigr] |\widehat f(\xi)|^2 &< \infty,
\end{aligned}\end{equation}
depending on the geometry.  Now by \eqref{HS bdd}, any $q\in H^s(\mathcal M)$ with $-\frac12<s\leq 0$ is an infinitesimally form bounded perturbation of $L_0$.  Consequently, by the KLMN~Theorem, the Lax operator $L= -i\partial_x +\frac1{2h}(e^{2ih\partial_x}-1)-q$ is self-adjoint, bounded from below, and has the same form domain as $L_0$; see~\cite{MR0493420} for further discussion.

The \eqref{KdV}, \eqref{BO}, and \eqref{ILW} equations are all specific examples of the more general Whitham equation, which combines the advection nonlinearity with a completely general linear dispersion:
\begin{equation}\label{BEQ}
\tfrac{d}{dt}q = - M q' - c q' - 2qq' .
\end{equation}
Here $M$ is the Fourier multiplier operator $\widehat{M f}(\xi) := m(\xi) \widehat f(\xi)$, associated to a symbol $m:\R\to\R$, and $c\in\R$ is a constant.  The background speed $c$ could be incorporated into the symbol $m$ or removed by the Galilei transformation $q\mapsto q -c/2$; however, we choose to keep it separate as in \eqref{ILW}.

As an offshoot of our calculations verifying the Lax pair formulation \eqref{Lax}, we have discovered a Lax pair representation of \eqref{BEQ} that we present in the next lemma.  We regard this Lax pair representation as illusory because the conjugacy class of this Lax operator \(\widetilde L(t)\) is completely independent of $q$; indeed,
\[
e^{-i\int_0^x q(t,y)\,dy}\widetilde L(t)e^{i\int_0^x q(t,y)\,dy} = -i\partial_x.
\]
Moreover, the renormalized perturbation determinant is always unity because the relevant integral operator is of Volterra type.

\begin{lemma}[Illusory Lax pair]\label{L:IllLP}
For a smooth trajectory $q:\R\to H^\infty(\mathcal M)$, let
\begin{equation}\label{faux LP}\begin{aligned}
\widetilde L(t) &:= -i\partial_x  - q(t), \\
\widetilde P(t) &:= -c \partial_x - i(Mq(t)) - i \partial_x^2   - \partial_x q(t) - q(t)\partial_x,
\end{aligned}\end{equation}
where $c\in \R$ and $M$ is a general Fourier multiplier.  Then
\begin{align}
\text{$q(t)$ solves \eqref{BEQ}} \iff \tfrac{d}{dt} \widetilde L(t) = [\widetilde P(t),\, \widetilde L(t)].    \label{IllLax}    
\end{align}
\end{lemma}

\begin{proof} We simply compute:
\begin{align*}
[\widetilde P,\, \widetilde L] &= c q'  + Mq' + i [\partial_x^2,q]  + i[\partial_x q + q \partial_x, \partial_x] + [\partial_x q + q \partial_x,q] \\
&= c q'  + Mq' + 2qq' .
\end{align*}
The operators on this last line are those of multiplication by the given functions.  The claim \eqref{IllLax} follows immediately.
\end{proof}

The claim in Lemma~\ref{L:IllLP} remains valid if $\widetilde P$ is replaced by the simpler operator
$$
\widetilde P(t)  - i \tilde L^2 = -c \partial_x - i(Mq(t)) - i q(t)^2;
$$
however, this would take us further from our main goal, namely, verifying \eqref{Lax}.

We now turn to verifying that \eqref{LP} is indeed a valid Lax pair for \eqref{ILW}.  By contrast with \eqref{faux LP}, the Lax pair \eqref{LP} does provide meaningful conserved quantities for \eqref{ILW}.  In fact, the next section is devoted to showing that these conserved quantities suffice to prove Theorem~\ref{T1}.  

\begin{theorem}\label{T:Lax}  For a smooth trajectory $q:\R\to H^\infty(\mathcal M)$,
\begin{align}
\text{$q(t)$ solves \eqref{ILW}} \iff \tfrac{d}{dt} L(t) = [P(t),\, L(t) ].    \label{Lax'}    
\end{align}
\end{theorem}

\begin{proof} Recall that $L$ and $P$ are defined in \eqref{LP}, and the Fourier multiplier $T$ is defined in \eqref{T defn}.

Applying Lemma~\ref{L:IllLP} with \(c=\frac1h\) and $M=T\partial_x$, we see that 
\begin{equation*}
\text{$q(t)$ solves \eqref{ILW}} \iff \tfrac{d}{dt} L(t) = [P(t)-\tfrac1h\partial_x e^{2ih\partial_x},\, L(t)- \tfrac1{2h}(e^{2ih\partial_x}-1)].  
\end{equation*}
Thus, it remains to demonstrate the vanishing of 
\begin{align*}
 - [P,\, \tfrac1{2h}e^{2ih\partial_x}] - [\tfrac1h\partial_x e^{2ih\partial_x},\, L] = [ i (Tq') + \partial_x q +q\partial_x ,\, \tfrac1{2h}e^{2ih\partial_x}] + [\tfrac1h\partial_x e^{2ih\partial_x},\, q].
\end{align*}

Let us consider this operator in Fourier variables.  On the line, we obtain an integral kernel that we write as $\tfrac{1}{\sqrt{2\pi}}K(\xi,\eta)$; on the circle, this is a matrix $K(\xi,\eta)$ indexed by $\xi,\eta\in2\pi\Z$.  In either case,
\begin{align*}
K(\xi,\eta)&=i(\xi-\eta)\coth\bigl(h[\xi-\eta]\bigr)\tfrac1{2h}\bigl[e^{-2h\xi} - e^{-2h\eta}\bigr] \widehat q (\xi-\eta)\\
&\quad - (i\xi + i\eta) \tfrac1{2h}\bigl[e^{-2h\xi} - e^{-2h\eta}\bigr] \widehat q (\xi-\eta) \\
&\quad + \tfrac1h\bigl[i\xi e^{-2h\xi} - i\eta e^{-2h\eta}\bigr] \widehat q (\xi-\eta).
\end{align*}
Using the identity
\begin{align*}
\coth\bigl(h[\xi-\eta]\bigr)\bigl[e^{-2h\xi} - e^{-2h\eta}\bigr] = - \bigl[ e^{-2h\xi} + e^{-2h\eta}\bigr],
\end{align*}
one readily checks that $K(\xi,\eta)\equiv 0$.
\end{proof}

\section{A priori bounds and equicontinuity}\label{S:3}

Our first task is to demonstrate the conservation of $\alpha(\kappa;q(t))$ under \eqref{ILW}:

\begin{prop}\label{P:cons}
Let $q(t)$ be a smooth solution to \eqref{ILW} and \(\kappa>0\) be chosen so that
\begin{equation}\label{small init}
\| \sqrt{R_0(\kappa)}\, q(0)\, \sqrt{R_0(\kappa)}\|_{\mathfrak I_2}^2<\tfrac1{36}.
\end{equation}
Then for all \(t\in \R\) we have
\begin{equation*}
\ddt \alpha(\kappa;q(t)) = 0.
\end{equation*}
Consequently,
\begin{equation}\label{I2 equiv}
\| \sqrt{R_0(\kappa)}\, q(t)\, \sqrt{R_0(\kappa)}\|_{\mathfrak I_2}^2\simeq \| \sqrt{R_0(\kappa)}\, q(0)\, \sqrt{R_0(\kappa)}\|_{\mathfrak I_2}^2.
\end{equation}
\end{prop}

\begin{proof}
Let \(I\) be the largest interval containing \(0\) for which
\[
\| \sqrt{R_0(\kappa)}\, q(t)\, \sqrt{R_0(\kappa)}\|_{\mathfrak I_2}^2<\tfrac19\quad\text{for all \(t\in I\)}.
\]
The interval \(I\) is readily seen to be non-empty by the hypothesis \eqref{small init} and open by continuity of the map \(t\mapsto \| \sqrt{R_0(\kappa)}\, q(t)\, \sqrt{R_0(\kappa)}\|_{\mathfrak I_2}^2\). Applying Lemma~\ref{L:C1}, we see that if \(\alpha(\kappa;q(t))\) is constant on \(I\) then for all \(t\in I\) we have
\begin{align*}
\| \sqrt{R_0(\kappa)}\, q(t)\, \sqrt{R_0(\kappa)}\|_{\mathfrak I_2}^2\leq 3\alpha(\kappa;q(t)) &= 3\alpha(\kappa;q(0))\\
&\leq 2\| \sqrt{R_0(\kappa)}\, q(0)\, \sqrt{R_0(\kappa)}\|_{\mathfrak I_2}^2<\tfrac1{18},
\end{align*}
which suffices to prove that the interval \(I\) is closed, and so comprises all \(t\in \R\).

It remains to prove that \(\alpha(\kappa;q(t))\) is constant on \(I\).  Applying Lemma~\ref{L:C1}, for \(t\in I\) we have
\begin{align}\label{dt KdValpha}
\tfrac{d }{dt} \alpha(\kappa;q(t)) = \sum_{\ell=2}^\infty \tr\Bigl\{ R_0\, \tfrac{d }{dt}q(t)\, \bigl(R_0 q(t)\bigr)^{\ell-1}\Bigr\}.
\end{align}
Here, we have used that \(R_0\frac d{dt}q,R_0q\in \I_2\) to cycle the trace. This is a consequence of \cite[Theorem 4.1]{MR2154153} together with the following facts: (i) as \(q\) is a smooth solution of \eqref{ILW}, we have \(q,\frac d{dt}q\in L^2(\mathcal M)\) and (ii) by Lemma~\ref{L:HS}, the symbol \(\frac1{a_h(\xi)+\kappa}\) of \(R_0(\kappa)\) satisfies
\[
\tfrac1{a_h(\xi)+\kappa}\lesssim \tfrac1{\frac{h\xi^2}{1+h|\xi|}+\kappa},
\]
which lies in \(L^2(\R)\) in the case \(\mathcal M = \R\) and in \(\ell^2(2\pi \Z)\) in the case \(\mathcal M = \T\).

Employing the equation \eqref{ILW} satisfied by $q$ and collecting like powers of $q$, we see that it suffices to show the following:
\begin{align}
\tr\bigl\{  R_0 (Tq'') R_0 \,q\bigr\} + \tr\bigl\{ R_0 \tfrac{q'}{h} R_0 \,q\bigr\}&=0, \label{E:quadratic}\\
\tr\Bigl\{ R_0  \bigl(Tq''+ \tfrac{q'}{h}\bigr)\bigl(R_0 q\bigr)^{l}\Bigr\}  +  \tr\Bigl\{ R_0\bigl(2 q q'\bigr)\bigl(R_0 q\bigr)^{l-1}\Bigr\}&=0, \quad\forall \ell\geq 2.\label{E:telescope}
\end{align}

We recall from Lemma~\ref{L:HS} that the function $\xi\mapsto F(\xi; \kappa,h)$ defined in \eqref{F defn} is even.  As we will now show, this is the key observation for proving \eqref{E:quadratic}.  For concreteness, we provide the details in the line setting. The circle case is a close parallel; one simply replaces integrals by sums.

A straightforward computation reveals that
\begin{align*}
\tr\bigl\{  R_0 (Tq'') R_0 \,q\bigr\} + \tr\bigl\{ R_0 \tfrac{q'}{h} R_0 \,q\bigr\}= \tfrac1{ih}\int_\R F(\xi;\kappa,h)  |\widehat q(\xi)|^2\bigl[ h\xi^2\coth(h\xi) - \xi\bigr]\, d\xi.
\end{align*}
As $q$ is real-valued, the function $\xi\mapsto |\widehat q(\xi)|^2$ is even.  However, $\xi\mapsto h\xi^2\coth(h\xi) - \xi$ is clearly an odd function.  Thus \eqref{E:quadratic} follows from the fact that $\xi\mapsto F(\xi; \kappa,h)$ is even.

We now turn to \eqref{E:telescope}. Writing $q'=[\partial_x,q]$, observing that $[\partial_x,R_0]=0$, and cycling the trace, we deduce
\begin{align}\label{easy}
\tr\Bigl\{ R_0  \bigl(\tfrac{q'}{h}\bigr)\bigl(R_0 q\bigr)^{l}\Bigr\}
&= \tfrac1h\tr\Bigl\{ R_0\, \partial_x \,q\bigl(R_0 q\bigr)^{l}\Bigr\}- \tfrac1h\tr\Bigl\{ R_0 \, q\,\partial_x\bigl(R_0 q\bigr)^{l}\Bigr\} =0.
\end{align}

To continue, we use the rules
\begin{align}\label{chain rule}
e^{2ih\partial_x}(gf) &= (e^{2ih\partial_x}g)\,(e^{2ih\partial_x}f), \\
(e^{2ih\partial_x} - 1) \,iT g &= (1+e^{2ih\partial_x})g,
\end{align}
to deduce that 
\begin{align*}
 [ e^{2ih\partial_x} , \,iTq']f &= \bigl(e^{2ih\partial_x}-1\bigr) (iTq') \cdot e^{2ih\partial_x} f
 	= (1+e^{2ih\partial_x})q' \cdot e^{2ih\partial_x} f \\
&= q' e^{2ih\partial_x} f + e^{2ih\partial_x}(q'f).
\end{align*}
This then yields the operator identity,
\begin{align}\label{444}
(Tq'') f &= [\partial_x, Tq'] f\\
&= \bigl[-i\partial_x + \tfrac1{2h}(e^{2ih\partial_x}-1) +\kappa, \,iTq'\bigr]f - \tfrac1{2h}[ e^{2ih\partial_x} , \,iTq']f \notag\\
&= \bigl[-i\partial_x + \tfrac1{2h}(e^{2ih\partial_x}-1) +\kappa,\, iTq'\bigr] f - \tfrac1{2h}q' e^{2ih\partial_x} f - \tfrac1{2h}e^{2ih\partial_x}(q'f),\notag
\end{align}
and correspondingly,
\begin{align*}
\tr\Bigl\{ R_0  \bigl(Tq''\bigr)\bigl(R_0 q\bigr)^{l}\Bigr\}
&= \tr\Bigl\{ \bigl(iTq'\bigr)\bigl(R_0 q\bigr)^{l}\Bigr\} -\tr\Bigl\{ R_0\bigl(iTq'\bigr)q\bigl(R_0 q\bigr)^{l-1}\Bigr\}\\
&\quad - \tr\Bigl\{ R_0\, \tfrac1{2h}q'e^{2ih\partial_x} \bigl(R_0 q\bigr)^{l}\Bigr\} -\tr\Bigl\{ R_0\, \tfrac1{2h}e^{2ih\partial_x} q' \bigl(R_0 q\bigr)^{l}\Bigr\}.
\end{align*}
Using $[q,  (iTq')]=0$ and cycling the trace, we see that the first two summands here cancel.  Employing the identity $q'=[\partial_x,q]$ and cycling the trace, we then compute as follows:
\begin{align*}
\tr\Bigl\{ R_0  \bigl(Tq''\bigr)\bigl(R_0 q\bigr)^{l}\Bigr\}
&= - \tr\Bigl\{ R_0\, \tfrac1{2h}q'e^{2ih\partial_x} \bigl(R_0 q\bigr)^{l}\Bigr\} -\tr\Bigl\{ R_0\, \tfrac1{2h}e^{2ih\partial_x} q' \bigl(R_0 q\bigr)^{l}\Bigr\}\\
&= - \tr\Bigl\{ R_0\, \tfrac1{2h}\partial_x \,q \,e^{2ih\partial_x} \bigl(R_0 q\bigr)^{l}\Bigr\} + \tr\Bigl\{ R_0\, \tfrac1{2h}q\,\partial_x \,e^{2ih\partial_x} \bigl(R_0 q\bigr)^{l}\Bigr\}\\
&\quad - \tr\Bigl\{ R_0\, \tfrac1{2h}e^{2ih\partial_x} \,\partial_x \,q \bigl(R_0 q\bigr)^{l}\Bigr\} +\tr\Bigl\{ R_0\, \tfrac1{2h}e^{2ih\partial_x}  \,q \,\partial_x\bigl(R_0 q\bigr)^{l}\Bigr\}\\
&=- \tr\Bigl\{ R_0\, \partial_x \,q \,\tfrac1{2h}e^{2ih\partial_x} \bigl(R_0 q\bigr)^{l}\Bigr\} +\tr\Bigl\{ R_0\, \tfrac1{2h}e^{2ih\partial_x}  \,q \,\partial_x\bigl(R_0 q\bigr)^{l}\Bigr\}\\
&= -\tr\Bigl\{ R_0\, \partial_x \,q^2 \bigl(R_0 q\bigr)^{l-1}\Bigr\} - \tr\Bigl\{ R_0\, \partial_x \,q \,(i\partial_x+\tfrac1{2h}-\kappa) \bigl(R_0 q\bigr)^{l}\Bigr\}\\
&\quad +\tr\Bigl\{ q \,\partial_x \bigl(R_0 q\bigr)^{l}\Bigr\} + \tr\Bigl\{ R_0\, (i\partial_x +\tfrac1{2h} -\kappa)  \,q \,\partial_x\bigl(R_0 q\bigr)^{l}\Bigr\}\\
&= -\tr\Bigl\{ R_0\, [\partial_x,q^2] \bigl(R_0 q\bigr)^{l-1}\Bigr\}.
\end{align*}
Claim \eqref{E:telescope} follows from this, the operator identity $ [\partial_x,q^2] =2qq'$, and \eqref{easy}.
\end{proof}

The preliminaries we have gathered so far almost suffice to prove all parts of Theorem~\ref{T1}.  There is just one exception, namely, equicontinuity in $L^2$.  For this result we need to combine the conservation of $\alpha(\kappa;q)$ with the conservation of momentum $M(q)$ defined in \eqref{moment}.  Our next lemma makes the connection between such a combination and $L^2$ equicontinuity.

\begin{lemma} \label{L:F} Let \(h_0>0\). The two functions $\xi\mapsto F(\xi; \kappa,h)$ defined in \eqref{F defn} satisfy
\begin{equation}\label{10:48}
\lim_{\kappa\to \infty} \kappa F(\xi; \kappa,h)= \tfrac 1{2\pi},
\end{equation}
uniformly for \(h\geq h_0\) and $\xi$ in compact subsets of $\R$, respectively \(2\pi\Z\).

Moreover, there are constants $A>0\) and $\kappa_1 = \kappa_1(h_0)>0$ so that if \(h\geq h_0\) then
\begin{align}\label{10:49}
 \kappa F(\xi; \kappa,h)\leq \tfrac1{4\pi}, \qtq{provided} \kappa\geq \kappa_1 \qtq{and} |\xi|\geq A\kappa.
\end{align}
Lastly, there is a constant $C = C(h_0)>0$ so that
\begin{align}\label{equi in}
 \bigl[ 1 + C\kappa^{-1/2}\bigr]M(q) - \pi \kappa \| \sqrt{R_0(\kappa)}\, q\, \sqrt{R_0(\kappa)}\|_{\mathfrak I_2}^2 \gtrsim  \|P_{\geq A\kappa} q\|_{L^2}^2 
\end{align}
for all \(h\geq h_0\) and $\kappa\geq \kappa_1$.  Here $P_{\geq N}$ denotes the sharp projection to frequencies $|\xi|\geq N$.
\end{lemma}

\begin{proof}
We proved in Lemma~\ref{L:HS} that the map \(\xi\mapsto F(\xi;\kappa,h)\) is even.  This allows us to consider only $\xi\geq 0$ for the remainder of the proof.  

We decompose $F(\xi;\kappa,h)=F_+(\xi;\kappa,h)+F_-(\xi;\kappa,h)$ by restricting the sum/integral defining $F$ to $\eta\geq0$ and $\eta<0$, respectively.  We will treat these pieces separately.  We begin with the former, introducing the notations
\begin{align*}
f(\eta,\xi;\kappa,h) = \tfrac1{(a_h(\eta) + \kappa) ( a_h(\eta+\xi) + \kappa )}
	\qtq{and} g(\eta,\xi;\vk) = \tfrac1{(\eta  + \vk)( \eta + \xi + \vk)},
\end{align*}
where \(a_h\) is defined as in \eqref{symbol}.

For $\xi,\eta\geq 0$ and \(\kappa>\frac1{2h}\), we clearly have
\begin{align}\label{10:50}
g(\eta,\xi;\kappa) \leq f(\eta,\xi;\kappa,h) \leq g\bigl(\eta,\xi;\kappa-\tfrac1{2h}\bigr);
\end{align}
moreover,
\begin{align}\label{10:51}
\int_0^\infty \vk \, g(\eta,\xi;\vk) \,\tfrac{d\eta}{2\pi} =  \tfrac\vk{2\pi\,\xi} \log\bigl[ 1 + \tfrac{\xi}{\vk} \bigr]  \longrightarrow \tfrac1{2\pi} \quad \text{as}\quad \vk\to\infty,
\end{align}
uniformly on compact sets of $\xi$.  In order to treat sums, we observe that due to the monotonicity of $\eta\mapsto g(\eta,\xi;\vk)$ and the fact that \(g(0,\xi;\vk) = \frac1{\vk(\xi+\vk)}\), we have
\begin{align}\label{10:52}
 \sum_{0\leq\eta\in2\pi\Z} g(\eta,\xi;\vk) - \tfrac1{\vk(\xi+\vk)}   \leq \int_0^\infty g(\eta,\xi;\vk) \,\tfrac{d\eta}{2\pi} \leq \sum_{0\leq\eta\in2\pi\Z} g(\eta,\xi;\vk);
\end{align}
indeed, we are comparing an integral with two Riemann sum approximations.

Combining all these observations yields that 
\begin{align}\label{F limit}
\kappa F_+(\xi; \kappa,h)\longrightarrow \tfrac 1{2\pi} \quad\text{as}\quad \kappa\to \infty
\end{align}
uniformly for $h\geq h_0$ and $\xi$ in compact subsets of $\R$ or $2\pi \Z$. Moreover, if \(\kappa>\frac1{2h}\) then in both geometries we may bound
\begin{align*}
\kappa F_+(\xi;\kappa,h) &\leq \tfrac \kappa {(\kappa - \frac1{2h})(\xi+\kappa-\frac1{2h})} + \tfrac\kappa{2\pi\,\xi} \log\bigl[ 1 + \tfrac{\xi}{\kappa - \frac1{2h}} \bigr].
\end{align*}
In particular, there exists \(A>0\) so that
\begin{align}\label{10:53}
\kappa F_+(\xi;\kappa,h)\leq \tfrac 2{\kappa A} + \tfrac1{2\pi A}\log(1+2A)\leq \tfrac1{8\pi}, \quad\text{provided \(\kappa\geq \tfrac1h\) and \(|\xi|\geq A\kappa\)}.
\end{align}
Moreover, we may bound
\begin{align}\label{10:53'}
 \kappa F_+(\xi; \kappa,h) - \tfrac1{2\pi}
 \leq \tfrac1\kappa\bigl(4 + \tfrac1{2\pi h}\bigr), \quad\text{for all \(\kappa\geq \tfrac1h\) and \(\xi\geq0\)}.
\end{align}

We turn now to $F_-$, which captures the contribution of $\eta<0$. Recalling that $a_h(\eta)\geq 0$ for all $\eta\in\R$ and noting that \(a_h(0) = 0 = a_h'(0)\) and \(a_h''(\eta) = 2he^{-2h\eta}\geq 2h\) for all \(\eta\leq 0\), we may bound
\[
f(\eta,\xi;\kappa,h) \leq \tfrac1{(h\eta^2 + \kappa)\kappa} \quad \text{for all} \quad\eta\leq 0. 
\]
It follows that in both geometries
\begin{align}\label{10:55}
 \kappa F_-(\xi; \kappa,h) \leq \int_0^\infty \tfrac1{h\eta^2 + \kappa}\,\tfrac{d\eta}{2\pi} = \tfrac1{4\sqrt{\kappa h}},
\end{align}
where we have again used the monotonicity of the map \(\eta\mapsto \tfrac1{(h\eta^2 + \kappa)\kappa} \) for the case that \(\mathcal M = \T\).

The claims \eqref{10:48} and \eqref{10:49} now follow by combining \eqref{F limit} through \eqref{10:55}.

Turning our attention to \eqref{equi in}, we notice that
\eqref{10:53'} and \eqref{10:55} ensure that for all \(h\geq h_0\) we have
$$
\sup_\xi \ 2\pi\kappa F(\xi;\kappa,h) \leq 1 + C \kappa^{-1/2} \quad\text{for $\kappa\geq \tfrac1{h_0}$},
$$
provided $C = C(h_0)>0$ is chosen sufficiently large.  Consequently, all frequencies $\xi\in\R$ make a non-negative contribution to LHS\eqref{equi in}.  Moreover, \eqref{10:49} ensures that frequencies $|\xi|\geq A\kappa$ contribute in the quantitative manner asserted.
\end{proof}

\begin{proof}[Proof of Theorem~\ref{T1}]
As the equation \eqref{ILW} is time-translation invariant, the a priori bounds \eqref{a priori bdds} will follow readily from the estimate
\begin{align}\label{a priori bdds'}
 \sup_{t\in \R}\| q(t) \|_{H^s} \lesssim_{s,h_0} \| q(0) \|_{H^s} \bigl( 1 +  \| q(0) \|_{H^s}\bigr)^{\frac{2|s|}{1-2|s|}}.
\end{align}
For $s=0$, this is a consequence of the conservation of momentum.  

Fix $-\frac12<s<0$ and $0<\delta\leq \frac16$.  Taking \(C_s>0\) as in \eqref{k bigger}, let
\[
\kappa_0(s,h_0)=\max\Bigl\{\tfrac1{h_0},\bigl[ 1 + \tfrac1{\delta^2} C_s \| q(0)\|_{H^s}^2 \bigr]^{\frac{1}{1-2|s|}}\Bigr\}.
\]
The estimate \eqref{small norm} now ensures that \eqref{small init} is satisfied and so Proposition~\ref{P:cons} applies.

For \(\kappa\geq \kappa_0\) and \(h\geq h_0\), the estimate \eqref{F main bound} shows that
\[
F(\xi;\kappa,h)\simeq
\tfrac1{\sqrt{\xi^2 + \kappa^2}}\log\bigl(4 + \tfrac{\xi^2}{\kappa^2}\bigr).
\]
Combining this with \eqref{I2}, \eqref{I2 equiv}, and the fact that
\begin{align*}
(\xi^2+\kappa^2)^s \simeq_s \int_{\kappa}^\infty  \log(4+\tfrac{\xi^2}{\varkappa^2})\frac{1}{\sqrt{\xi^2+\varkappa^2}} \,\varkappa^{2s} \,d\varkappa,
\end{align*}
it follows that for all \(\kappa\geq \kappa_0\) and \(h\geq h_0\) we have
\begin{align}\label{H kappa bdds}
\sup_{t\in\R}\, \|q(t)\|_{H^s_{\kappa}}^2 \lesssim_s \| q(0) \|_{H^s_{\kappa}}^2.
\end{align}
In particular, by taking $\kappa=\kappa_0$, we obtain
\begin{equation*}
\sup_{t\in\R}\, \|q(t)\|_{H^s} \lesssim_s \kappa_0^{|s|} \| q(0) \|_{H^s_{\kappa_0}}\lesssim_{s,h_0} \| q(0) \|_{H^s} \bigl( 1 +  \| q(0) \|_{H^s}\bigr)^{\frac{2|s|}{1-2|s|}},
\end{equation*}
which completes the proof of \eqref{a priori bdds'} in the case $-\frac12<s<0$.

We now turn to the second statement in Theorem~\ref{T1}, namely, the equicontinuity of orbits. When $-\frac12<s<0$, this claim follows readily from \eqref{H kappa bdds} and Lemma~\ref{L:equi}.

It remains to prove the equicontinuity of orbits in $L^2(\mathcal M)$. This is the role of Lemma~\ref{L:F}. We will provide the details in the line setting.

Given $q_0\in Q$, let $q(t)$ denote the solution to \eqref{ILW} with initial data $q(0)=q_0$.  Using the conservation of both momentum and $\alpha$ under the \eqref{ILW} flow, together with Lemmas~\ref{L:dumb}, \ref{L:HS}, and~\ref{L:F}, we may bound
\begin{align*}
M(q&(t)) - \pi\kappa \bigl\| \sqrt{R_0(\kappa)}\, q(t)\, \sqrt{R_0(\kappa)}\bigr\|_{\mathfrak I_2(\R)}^2 \notag\\
&\lesssim \Bigl| M(q(t)) - 2\pi\kappa \alpha(\kappa;q(t))\Bigr|
	+ \sum_{\ell\geq 3} \kappa \bigl\| \sqrt{R_0(\kappa)}\, q(t)\, \sqrt{R_0(\kappa)}\bigr\|_{\mathfrak I_2(\R)}^\ell \notag\\
&\lesssim \Bigl| M(q_0) - 2\pi\kappa \alpha(\kappa;q_0)\Bigr|
	+ \sum_{\ell\geq 3} \kappa \bigl\| \sqrt{R_0(\kappa)}\, q(t)\, \sqrt{R_0(\kappa)} \bigr\|_{\mathfrak I_2(\R)}^\ell \notag\\
&\lesssim \biggl| \int_{-\infty}^\infty \bigl[\tfrac12-\pi\kappa F(\xi;\kappa,h)\bigr]|\widehat q_0(\xi)|^2\,d\xi \biggr|
	+ \sum_{\ell\geq 3} \kappa\bigl\| \sqrt{R_0(\kappa)}\, q(t)\, \sqrt{R_0(\kappa)}\bigr\|_{\mathfrak I_2(\R)}^\ell \notag\\
&\quad + \sum_{\ell\geq 3} \kappa \bigl\| \sqrt{R_0(\kappa)}\, q_0\, \sqrt{R_0(\kappa)}\bigr\|_{\mathfrak I_2(\R)}^\ell \notag\\
&\lesssim \int_{-\infty}^\infty \bigl|1-2\pi\kappa F(\xi;\kappa,h)\bigr||\widehat q_0(\xi)|^2\,d\xi
	+ \sum_{\ell\geq 3} \kappa \bigl[C\kappa^{-1} M(q_0) \bigr]^{\ell/2}  \notag\\
&\lesssim \int_{-\infty}^\infty \bigl|1-2\pi\kappa F(\xi;\kappa,h)\bigr| |\widehat q_0(\xi)|^2\,d\xi + \kappa^{-\frac12} M(q_0)^{\frac32},
\end{align*}
uniformly for $t\in \R$ and \(h\geq h_0\), provided $\kappa\geq\frac1{h_0}$ is chosen sufficiently large depending on $M(q_0)$.  Combining this with \eqref{equi in}, we obtain
\begin{align*}
\sup_{t\in\R} \|P_{\geq A\kappa} q(t)\|_{L^2}^2 &\lesssim_{h_0} \int_{-\infty}^\infty \bigl|1-2\pi\kappa F(\xi;\kappa,h)\bigr| |\widehat q_0(\xi)|^2\,d\xi
	+ \kappa^{-\frac12} M(q)^{\frac32} + \kappa^{-\frac12} M(q),
\end{align*}
for sufficiently large $\kappa$, depending only on $h_0$ and the norm bound on the set $Q$.  Our proof will be finished once we show that the right-hand side converges to zero as $\kappa\to\infty$, uniformly for $q\in Q$ and \(h\geq h_0\).  This duty is easily dispatched:  As $Q$ is equicontinuous in $L^2$,
$$
\lim_{N\to \infty} \,\sup_{q\in Q}\,\|P_{\geq N}q_0\|_{L^2}=0.
$$
On the other hand, as $Q$ is bounded in $L^2$, \eqref{10:48} implies
$$
\lim_{\kappa\to \infty} \,\sup_{h\geq h_0}\,\sup_{q\in Q}\, \int_{|\xi|\leq N} \bigl| 1-2\pi\kappa F(\xi;\kappa,h)\bigr| |\widehat q(\xi)|^2\,d\xi =0
$$
for each fixed $N>0$.
\end{proof}

Employing a similar argument, we conclude with the

\begin{proof}[Proof of Theorem~\ref{T2}]
Let \(C_s>0\) be defined as in \eqref{k bigger}, \eqref{shallow k bigger}. Given \(K>0\), we choose sufficiently small \(h_0 = h_0(s,K)>0\) to ensure that
\[
\tfrac1{h_0}\geq \bigl(1 + \tfrac{h_0^2}{36}C_sK^2\bigr)^{\frac1{1-2|s|}}.
\]
Employing \eqref{small norm}, we then see that whenever \(\|u(0)\|_{H^s}\leq K\) and \(0<h\leq h_0\) the estimate \eqref{small init} is satisfied for all \(\kappa>\frac1h\).  We wish to consider the potentially larger range $\kappa\geq \mu_0^2h$ with
\[
\mu_0^2 = \bigl[ 1 + \tfrac1{36} C_s \|u(0)\|_{H^s}^2 \bigr]^{\frac{1}{\frac32-|s|}},
\]
as on LHS\eqref{shallow k bigger}.  This choice of $\mu_0$ ensures that  \eqref{small init} remains true for $\mu_0^2h\leq\kappa\leq \frac1h$ and  \(0<h\leq h_0\).

From \eqref{F main bound}, we see that
\[
F(\xi;\kappa,h)\simeq \begin{cases}\displaystyle\bigl[\tfrac{h\xi^2}{1+h|\xi|} +\kappa\bigr]^{-1}\bigl[\tfrac{1}{\sqrt{h\kappa}}+ \log\bigl(4 + h^2\xi^2\bigr)\bigr]&\text{if \(0<\kappa\leq \frac1h\),}\vspace{1em}\\\displaystyle\frac1{\sqrt{\xi^2 + \kappa^2}}\log\bigl(4 + \tfrac{\xi^2}{\kappa^2}\bigr)&\text{if \(\kappa>\frac1h\).}\end{cases}
\]
Given \(\mu\geq \mu_0\) and \(0<h\leq h_0\), we then either have \(\mu > \frac1h\), in which case
\[
\int_\mu^\infty F(\xi;\kappa,h)\tfrac{\kappa^{2s}}{h^2}\,d\kappa\simeq_s \tfrac1{h^2}(\xi^2 + \mu^2)^s,
\]
or we have \(\mu\leq \frac1h\), in which case
\[
\int_{\mu^2h}^{\frac1{h}} F(\xi;\kappa,h) \bigl(\tfrac1 h\bigr)^{s+\frac32} \kappa^{s+\frac12} \,d\kappa + \int_{\frac1{h}}^\infty F(\xi;\kappa,h) \tfrac{\kappa^{2s}}{h^2} \,d\kappa \simeq_s \tfrac1{h^2}(\xi^2 + \mu^2)^{s},
\]
where the first integral is readily bounded by considering the cases \(|\xi|\leq \mu\), \(\mu<|\xi|\leq \frac1h\), and \(|\xi|>\frac1h\) separately.

Applying Proposition~\ref{P:cons} as in the proof of Theorem~\ref{T1}, we obtain the estimate
\[
\sup_{t\in \R}\|u(t)\|_{H_\mu^s}^2\lesssim_s \|u(0)\|_{H_\mu^s}^2,
\]
for all \(\|u(0)\|_{H^s}\leq K\), \(0<h\leq h_0\), and \(\mu\geq \mu_0\). The a priori bounds \eqref{a priori bdds''} now follow from taking \(\mu =\mu_0\), and the equicontinuity from applying Lemma~\ref{L:equi}.
\end{proof}

\end{document}